\documentclass{article}
\usepackage[utf8]{inputenc}
\usepackage[english]{babel}

\usepackage{amssymb}
\usepackage{mathtools}
\usepackage{tikz}
\usepackage{tikz-cd}
\usepackage{mathtools}
\usepackage{amsmath}
\usepackage{amsthm}
\usepackage{hyperref}

\usepackage{graphics}
\usepackage{color}
\usepackage{tabularx,colortbl}
\usetikzlibrary{positioning}
\usetikzlibrary{decorations.markings,arrows}
\usetikzlibrary{calc}
\usepackage{amsthm}

\newcounter{thrm}
    \newtheorem{trm}[thrm]{\bf Theorem}
	
	\newtheorem{prop}[thrm]{\bf Proposition}
	\newtheorem{rmk}[thrm]{\bf Remark }
	
	\newtheorem{cor}[thrm]{\bf Corollary}
	\newtheorem{lm}[thrm]{\bf Lemma}

	\newtheorem{defi}[thrm]{\bf Definition}

\title{The image in $\mathcal{M}_g$ of strata of meromorphic and quadratic differentials}
\author{Andrei Bud}
\date{}

\begin{document}

\maketitle
\begin{abstract}
	We compute the dimension of the image of the map $\pi_Z \colon Z \rightarrow \mathcal{M}_g$ forgetting the markings, where $Z$ is a connected component of the stratum $\mathcal{H}^k_g(\mu)$ of $k$-differentials with an assigned partition $\mu$, for the cases when $k=1$ with meromorphic partition and $k=2$ when the quadratic differentials have at worst simple poles.  
\end{abstract}


\section*{Introduction} \ \par

For positive integers $g$ and $k$, let $\mu =(m_1,\ldots,m_n)$ a partition of $(2g-2)k$ and define $\mathcal{H}_g^k(\mu) \subseteq \mathcal{M}_{g,n}$ to be the moduli space of $k$-canonical divisors of type $\mu$, parametrizing pointed curves $[C,p_1,\ldots,p_n]$ satisfying $\mathcal{O}_C(\sum_{i=1}^n m_i p_i) \cong \omega_C^k$. It is natural to consider a connected component $Z$ of $\mathcal{H}_g^k(\mu)$ and ask what is the dimension of the image of the forgetful map $\pi_Z\colon Z  \rightarrow \mathcal{M}_g$. We answer this question in the cases $k=1$ with meromorphic partition and $k=2$ when the quadratic differentials have at worst simple poles. Consequently, our results will provide new divisors along with higher
codimension cycles on the moduli space $\mathcal{M}_g$. To underline the importance of such cycles we point out that for $k = 1$ and $\mu = (2,2,\ldots,2)$ the image of the even component is the divisor of curves with a vanishing theta null, see \cite{teix}. \\

We start with the case $k=1$ and drop the superscript $k$ from the notation of the moduli of canonical divisors. It is obvious that if $\mu$ has a unique negative entry, equal to $-1$, the stratum $\mathcal{H}_g(\mu)$ is empty. We will assume in what follows that the partition $\mu$ is not of this form and prove the following result:

\begin{trm}\label{T1}
	For $g\geq 2$, let $\mu$ be a strictly meromorphic partition  of $2g-2$ of length $n$ and $Z$ a connected, non-hyperelliptic component of $\mathcal{H}_g(\mu)$. Then the dimension of the image of the forgetful map $\pi_Z\colon Z \rightarrow \mathcal{M}_g$ is the expected one, that is, $\min\left\{2g+n-3, 3g-3\right\}$. 
\end{trm} 

 Together with the case of holomorphic differentials, treated by Gendron in \cite{Gen18}, and the obvious case of hyperelliptic components, Theorem \ref{T1} completely answers the question for strata with $k=1$. As a consequence, every stratum $\mathcal{H}_g(\mu)$ has a connected component $Z$ whose dimension of the image in $\mathcal{M}_g$ is equal to $\min\left\{\dim(Z),3g-3\right\}$. \\
 
  We prove a similar theorem for the case $k=2$ when the poles are at worst of order $1$. When we have $\mu = (2m_1,\ldots,2m_n)$ a positive partition of $4g-4$, the stratum $\mathcal{H}_g^2(\mu)$ contains the components of $\mathcal{H}_g(\frac{\mu}{2})$. We will denote by $\mathcal{Q}_g(\mu)$ the union of connected components of $\mathcal{H}_g^2(\mu)$ that are not components of $\mathcal{H}_g(\frac{\mu}{2})$. When $\mu$ has at least one odd entry we make the convention $\mathcal{Q}_g(\mu) = \mathcal{H}_g^2(\mu)$. 
 
 If $\mu$ is a partition of $4g-4$ as above, there are four cases when $\mathcal{Q}_g(\mu)$ is empty; namely $g=1$ and $\mu = (1,-1)$, or $\mu = (0)$ and the cases when $g=2$ and $\mu = (3,1)$ or $\mu = (4)$. We assume in the next theorem that $\mu$ is not one of these partitions. \\
 
 In this case, almost any non-hyperelliptic component of a stratum has dimension of the image in $\mathcal{M}_g$ equal to $\min\left\{2g-3+n,3g-3\right\}$ with a unique exception in genus 4. It was proven by Chen and M\"oller in \cite{ChenM} that if $\mu = (3,3,3,3)$, the stratum $\mathcal{Q}_4(3,3,3,3)$ has two non-hyperelliptic connected components distinguished by whether the value $h^0(C,p_1+p_2+p_3+p_4)$ of a pointed curve $[C,p_1,p_2,p_3,p_4]$ in the stratum is 1 or 2. As this value plays a fundamental role in our proof, we get that the component $\mathcal{Q}_4^{\mathrm{Irr} }(3,3,3,3)$  with associated value 2, has dimension of the image one less than expected. We prove the following:
 
 \begin{trm}\label{T2} Let $g \geq 2 $ and $\mu$ be a partition of $4g-4$ with entries either positive or $-1$. Then if $Z$ is a connected component of $\mathcal{Q}_g(\mu)$ that is not hyperelliptic, the dimension of the image of the forgetful map $ \pi_Z \colon Z \rightarrow \mathcal{M}_g$ is $\min\left\{2g-3+n,3g-3 \right\}$, with the unique exception $Z = \mathcal{Q}_4^{ \mathrm{Irr} }(3,3,3,3)$ in genus $4$, when the dimension of the image is $2g-4+n = 8$.
 \end{trm} \par

  Another case for which the answer is known is when the length of the partition $\mu$ is at least $g \geq 3$ and $\mathcal{H}_g^k(\mu)$ has a unique irreducible component, see \cite{Barthesis} and \cite{BAR18}. This case will be used to simplify the proofs of the two theorems. In this paper we rely on the description of the forgetful map $\pi_{\mu}:\mathcal{H}_g^k(\mu)\rightarrow \mathcal{M}_g$ at the level of tangent spaces appearing in \cite{Daweik-diffcomp}, \cite{Mon} and on a degeneration argument that ensures the locus where the tangent map is injective is non-empty.

It is known that the components of the Deligne-Mumford strata $\mathcal{H}_g^k(\mu)$ are smooth of dimension $2g-2+n$ when $k=1$ and $\mu$ is holomorphic and $2g-3+n$ if the component is not one corresponding to $k$-th powers of holomorphic abelian differentials. We refer the reader to \cite{Daweik-diffcomp} and \cite{Sch18} for an account of these results using deformation theory. 

The Deligne-Mumford compactification of the strata $\mathcal{H}_g^k(\mu)$ will play an important role in this article, as we will often degenerate to a singular curve in the boundary of $\mathcal{M}_{g,n}$ inside $\overline{\mathcal{M}}_{g,n}$. Results in this direction appear in \cite{FP18}, where it is shown that the compactification is a component of the space of twisted canonical divisors. An exhaustive description of the curves in the boundary of the strata was achieved in \cite{DaweiAbComp} for $k=1$ and \cite{Daweik-diffcomp} for $k$ at least 2. 

The last important ingredient we need is the description of the connected components of $\mathcal{H}_g^k(\mu)$. Unfortunately, the answer is not known for $k\geq 3$. The case of holomorphic abelian differentials was treated by Kontsevich and Zorich in \cite{abelcompo}, the meromorphic case was treated by Boissy in \cite{Boissy} and the case of $k=2$ when the partition $\mu$ has entries greater or equal to $-1$ was studied by Lanneau in \cite{Lanneau}. Lanneau's list of sporadic strata was found to be inexhaustive and was completed by Chen and M\"oller in \cite{ChenM}. We know that in the case $k=1$, additional components appear because of the hyperelliptic and spin structure while for $k=2$ with entries at least $-1$, such components appear due to hyperellipticity.  \newline
\textbf{Acknowledgement} I would like to thank my advisor Gavril Farkas for all the support he offered me along the way and for introducing me to this beautiful topic.  I am also grateful to Ignacio Barros and Dawei Chen for contributing with insightful ideas. I am also thankful to the referee for the several suggested improvements.

\section{General approach} \ \par
Let $\mu = (m_1,\ldots,m_n)$ a partition of $(2g-2)k$ of length $n$. We consider the Deligne-Mumford substack $\mathcal{H}_g^k(\mu) \subset \mathcal{M}_{g,n}$ given by 
\[ \mathcal{H}_g^k(\mu)= \left\{[C,p_1,\ldots,p_n] \in \mathcal{M}_{g,n} \mid \mathcal{O}_C(\sum_{i=1}^n m_ip_i) \cong \omega_C^k \ \right\} \] \par

We see from Theorem \ref{T1} and Theorem \ref{T2} that the majority of connected components have image of dimension $\min\left\{2g-3+n,3g-3\right\}$ in $\mathcal{M}_g$. As we will encounter such components a lot, we make the convention: 
\begin{defi}
	Let $Z \subseteq \mathcal{H}^k_g(\mu)$ a connected component of the stratum and $\pi_Z\colon Z \rightarrow \mathcal{M}_g$ the forgetful map. We say that $Z$ is of the expected image dimension if the image of the forgetful map $\pi_Z$ has dimension equal to $\min\left\{\dim(Z), 3g-3\right\}$.
\end{defi}

The next proposition is fundamental in the proofs of Theorems \ref{T1} and \ref{T2}, as it offers information about the behaviour of the map from $\mathcal{H}^k_g(\mu)$ forgetting any of the sections.  \par 
\begin{prop}\label{fibdim}
	Let $\mu = (m_1,\ldots,m_n)$ a partition of $(2g-2)k$ with non-zero entries, $A\subseteq \left\{1,2,\ldots,n \right\}$ of cardinality denoted by $n'$ and $B$ the complement of $A$. Consider the composition map $\pi_{\mu,B}$ given by 
	\[\mathcal{H}^k_g(\mu)\xhookrightarrow{i}\mathcal{M}_{g,n} \xrightarrow{\pi} \mathcal{M}_{g,n'}\]
	where $\pi$ is the map forgetting the marked points $p_i$ for $i\in B$. \par
For a point $[C,p_1,\ldots,p_n] \in \mathcal{H}^k_g(\mu)$, the map at the level of tangent spaces 
\[ d\pi_{\mu,B}\colon T_{[C,p_1,\ldots,p_n]}\mathcal{H}_g^k(\mu)\rightarrow T_{[C,\left\{p_i\right\}_{i\in A}]} \mathcal{M}_{g,n'} \] 
has $d$-dimensional kernel if and only if $h^0(C, \sum\limits_{i\in B}p_i) = d+1$ 
\end{prop}
\begin{proof}
	Let $[C,p_1,\ldots,p_n] \in \mathcal{H}_g^k(\mu)$, together with a $k$-differential $\varphi$ satisfying \[\mathrm{div}(\varphi) = \sum_{i=1}^n m_ip_i \] \par
We will now use the characterization in \cite{Daweik-diffcomp} and \cite{Mon} of the tangent space of $\mathcal{H}_g^k(\mu)$ at a point. Take the map:
 \begin{align*}
\beta:T_C(-\sum_{i=1}^n p_i) &\rightarrow \omega_C^{ k}(-\sum_{i=1}^n m_i p_i) \\
w &\mapsto \frac{d(\varphi \cdot w^{k})}{w^{ k-1}}
\end{align*} \par
The tangent space $T_{[C,p_1,\ldots,p_n]}\mathcal{H}_g^k(\mu)$ is identified with $\text{ker}(H^1(\beta))$ and the tangent map $d\pi_{\mu,B}$ is the dashed map in the diagram 
\[
\begin{tikzcd}
0 \arrow[r, ""] & T_{[C,p_1,\ldots,p_n]} \mathcal{H}_g^k(\mu)  \arrow{r}{ } \arrow[swap, dashed]{dr}{d\pi_{\mu,B}} & H^1(C, T_C(-\sum\limits_{i=1}^n p_i )) \arrow{d}{} \arrow{r}{H^1(\beta)}  & H^1(C,\omega_C^k(-\sum\limits_{i=1}^n m_ip_i ) ) \\
&  &  H^1(C, T_C(-\sum\limits_{i\in A}p_i)) &
\end{tikzcd}
\] 
where the vertical map is the one induced by the injective morphism \[T_C(-\sum\limits_{i=1}^n p_i) \rightarrow T_C(-\sum\limits_{i\in A}p_i)\] \par 
By Serre duality, the dashed map above has $d$-dimensional kernel if and only if the dashed map in the next diagram has $d$-dimensional cokernel: 
 \[
\begin{tikzcd}
H^0(C,\omega_C^{1-k}(\sum\limits_{i=1}^n m_i p_i))  \arrow{r}{H^0(\beta^\vee)}  & H^0(C, \omega_C^2(\sum\limits_{i=1}^n p_i ))  \arrow{r}{b}  & \text{ \ coker}(H^0(\beta^{\vee})) \arrow[r, ""] & 0 \\
&  H^0(C, \omega_C^2(\sum\limits_{i\in A}p_i)) \arrow{u}{a} \arrow[swap, dashed]{ur}{d\pi_{\mu,B}^{\vee}}  & &
\end{tikzcd}
\] \par 
The map $\beta^{\vee}:\omega_C^{ 1-k}(\sum_{i=1}^n m_i p_i) \rightarrow \omega_C^{2} (\sum_{i=1}^n p_i ) $ is given by: 
\[ \beta^{\vee}(s) = -\frac{d(\varphi^{ k-1} s^{ k} )}{\varphi^{ k-2}  s^{ k-1}}   \] \par 
We consider the case when $\varphi$ is not the $k$-th power of a holomorphic abelian differential. The same method with slight and obvious modifications can be applied to the holomorphic abelian case. \par
Let $a$ and $b$ the maps labelled as such in the second diagram. As the map $a$ is injective, we use the exactness of the diagram to deduce that 
 $\mathrm{coker}(b \circ a )$ is $d$-dimensional if and only if 
 \[\dim (\mathrm{Im}(H^0(\beta^{\vee})) \cap \mathrm{Im}(a)) = g+n'+d-n \] \par 
 The problem translates into understanding when elements of the space $ H^0(C,\omega_C^{1-k}(\sum_{i=1}^n m_i p_i))$ are mapped to quadratic differentials that are holomorophic outside the points $p_i$ for $i\in A$. \\
 
 Let $s \in H^0(C,\omega_C^{1-k}(\sum_{i=1}^n m_i p_i)) $ and denote by $-n_i$ its order of vanishing at $p_i$, where the order is negative for poles. We also have from the definition that the $k$-differential $\varphi$ has order $-m_i$ at $p_i$. It is a well known fact that if a function $f$ has order $n \neq 0$ at a point $p$, then the differential $df$ has order $n-1$ at $p$. It follows from this discussion and from the definition of  $\beta^{\vee}$ that $\beta^{\vee}(s)$ has order at $p_i$ equal to $m_i -n_i -1$ if we have $kn_i-(k-1)m_i \neq 0$ or at least $m_i -n_i$ if  $kn_i-(k-1)m_i = 0$ \\

 As $s \in H^0(C,\omega_C^{1-k}(\sum_{i=1}^n m_i p_i)) $ it follows that $m_i \geq n_i$ for all $i = 1,\ldots,n$. The quadratic form $\beta^{\vee}(s)$ is holomorphic at the point $p_j$ for $j \in B$ if and only if we have $n_j \leq m_j - 1$. It follows that $s$ is sent by $H^0(\beta^{\vee})$ to $\mathrm{Im}(a)$ if and only if 
 \[ s\in H^0\text{\large(}C,\omega_C^{1-k}(\sum\limits_{i=1}^n m_i p_i - \sum\limits_{i \in B}p_i )\text{\large)} \] \par 
 As $H^0(\beta^{\vee})$ is injective it follows that:   
 \[\dim\text{\large(}\mathrm{Im}(H^0(\beta^{\vee})) \cap \mathrm{Im}(a) \text{\large)} = h^0\text{\large(}C,\omega_C^{1-k}(\sum\limits_{i=1}^n m_i p_i - \sum\limits_{i \in B}p_i )\text{\large)} \] \par 
 The conclusion follows, as the Riemann-Roch Theorem implies that the right-hand-side is $(g+n'+d-n)$-dimensional if and only if 
 \[h^0(C,\sum\limits_{i\in B}p_i) = d+1\] \par
 \end{proof} 

Using the well-known relation between the dimension of the generic fiber and the rank of the generic tangent map, we deduce as a straight forward application of Proposition \ref{fibdim} the following:
\begin{cor}\label{cor}
	Let $Z$ be a connected component of $\mathcal{H}_g^k(\mu)$ and $A \subseteq \left\{ 1,2,\ldots,n\right\}$. Then the dimension of the generic fiber of the map $\pi_{Z,A}\colon Z\rightarrow \mathcal{M}_{g,n'}$ forgetting the sections $\left\{p_i\right\}_{i\in A}$ is equal to $d-1$ where
	\[ d = \min\limits_{[C,p_1,\ldots,p_n] \in Z} \left\{h^0(C,\sum\limits_{i\in A} p_i) \right\}  \]   
\end{cor} \ \par \ \par

As it is not clear how to find smooth $n$-pointed curves $[C,p_1,\ldots,p_n]$ satisfying $h^0(C,p_1+\ldots+p_n)= d$, we degenerate to a nodal curve with one node and two irreducible components. We explain below, using the Riemann-Roch theorem, how this condition degenerates in this case. \\

By taking the long exact sequence associated to the short exact sequence
\[ 0 \rightarrow \omega_C(-\sum\limits_{i\in A} p_i) \rightarrow \omega_C \rightarrow  \omega_C/ \omega_C(-\sum\limits_{i\in A} p_i) \rightarrow 0 \]
we get that $h^0(C,\sum\limits_{i\in A} p_i) = d$ is equivalent to the map
\[H^0(C,\omega_C) \rightarrow H^0(C,\omega_C/ \omega_C(-\sum\limits_{i\in A} p_i)) \]
being of rank $1+n'-d$, where $n'$ denotes the cardinality of $A$. \par 
This observation is essential as this equivalent form is a rank condition for a vector bundle morphism and can be extended to $\overline{\mathcal{M}}_{g,n}$. This paves the way for using a degeneracy argument to deduce the dimension of a generic fiber of $\pi$. \\

Consider a curve $C$ = $C_1 \cup C_2$ consisting of two irreducible components glued together at exactly one node. We know that 
\[ H^0(C,\omega_C) = H^0(C_1,\omega_{C_1})\oplus H^0(C_2,\omega_{C_2}) \]
In particular, by denoting $A_i = \left\{ j\in A \mid p_j \in C_i \right\}$ we see that 
\[ H^0\text{\large(}C,\omega_C/ \omega_C(-\sum\limits_{i\in A} p_i)\text{\large)} = H^0\text{\large(}C_1,\omega_{C_1}/ \omega_{C_1}(-\sum\limits_{i\in A_1} p_i)\text{\large)} \oplus H^0\text{\large(}C_2,\omega_{C_2}/ \omega_{C_2}(-\sum\limits_{i\in A_2} p_i)\text{\large)} \] \par
We will only be interested in the case $d=1$, where we have the following: 
\begin{lm}\label{lmdeg}
	With the above notations the map $H^0(C,\omega_C) \rightarrow H^0(C,\omega_C/ \omega_C(-\sum\limits_{i\in A} p_i))$ is surjective if and only if the two maps 
	\[ \normalfont H^0(C_1,\omega_{C_1}) \rightarrow H^0\text{\large(}C_1,\omega_{C_1}/ \omega_{C_1}(-\sum\limits_{i\in A_1} p_i)\text{\large)} \text{ \ and \ } H^0(C_2,\omega_{C_2}) \rightarrow H^0\text{\large(}C_2,\omega_{C_2}/ \omega_{C_2}(-\sum\limits_{i\in A_2} p_i)\text{\large)} \]
	are surjective.
\end{lm}

\section{Meromorphic strata} \ \par
As explained previously, the condition $h^0(C,\sum_{i\in A} p_i)= 1$ can be interpreted as a maximal rank condition for a morphism of vector bundles on $\overline{\mathcal{M}}_{g,n}$, hence an open condition. To deduce that this condition is generically satisfied on a component $Z$ of $\mathcal{H}_g^k(\mu)$ it is enough to find a nodal curve $C$ on $\overline{Z}$ respecting the surjectivity in Lemma \ref{lmdeg}. The lemma implies we can use induction on the genus to deduce such conditions. We will start with the case $g=2$, which will serve as the initial step in our treatment. 
\begin{prop} \label{k=1g=2} Let $\mu$ a strictly meromorphic partition of 2 of length $n$ such that $\mathcal{H}_2(\mu)$ is non-empty. Then for every subset $A\subseteq\left\{1,2\ldots,n\right\}$ of cardinality $|A| \leq 2$, there exists an $n$-pointed curve $[C,p_1,\ldots,p_n] \in \mathcal{H}_2(\mu)$ satisfying  
\[h^0(C,\sum_{i\in A} p_i)= 1\]	
\end{prop}
\begin{proof}
When $|A| = 1$ the result is obvious. We will consider the case $|A| =2$ . \par
Take first the case when there exists $i \in A$ such that $m_i \geq 2$. Then, we can take the clutching along the last markings: 
\[\mathcal{H}_1(m_i,-m_i)\times \mathcal{H}_1(m_1,\ldots,\hat{m}_i,\ldots,m_n,m_i-2)\rightarrow \overline{\mathcal{M}}_{2,n}\]
The image of this morphism is in $\overline{\mathcal{H}}_2(\mu)$ by \cite{DaweiAbComp}, hence the conclusion follows by Lemma \ref{lmdeg}. \par
When both markings in $A$ correspond to elements $m_i \leq 1$, take $B \subseteq \left\{1,2,\ldots,n\right\}$ consisting of the elements $j$ in the complement of $A$ such that $m_j > 0 $. Without any loss of generality assume that $A = \left\{1,2\right\}$ and $m_1 \geq m_2$. We consider the clutching along the last markings 
\[ \mathcal{H}_1(m_1,\left\{m_j\right\}_{j\in B}, -m_1-\sum_{j \in B}m_j)\times \mathcal{H}_1(m_2,\left\{m_j\right\}_{j\notin B\cup \left\{1,2\right\}},m_1+\sum_{j \in B}m_j - 2)\rightarrow\overline{\mathcal{H}}_2(\mu) \subseteq \overline{\mathcal{M}}_{2,n} \] \par
The conclusion follows again by Lemma \ref{lmdeg}. The only case when we cannot apply the lemma is when one of the two strata is empty. This happens only if $m_1=m_2 = -1$ and all other entries are positive. Assume in this case that $p_1+p_2 = g^1_2$ for all curves $[C,p_1,\ldots,p_n] \in \mathcal{H}_2(\mu)$. We observe that the image of the map forgetting $p_1$ and $p_2$ is in $\mathcal{H}_2^2(m_3,\ldots,m_n)$ and the generic fiber is one-dimensional. \par By dimension reasons, it follows that the image must be contained in a component corresponding to abelian holomorphic differentials. This is possible if and only if $m_3,\ldots,m_n$ are all even. Hence we are left with the cases $\mu = (-1,-1,2,2)$ and $\mu = (-1,-1,4)$, and we see that in the nonhyperelliptic component, it is not true that $p_1+p_2$ is always a $g^1_2$.
\end{proof} 

We are now ready to prove the statement for all genera. 
\begin{prop}  \label{gen}
	Let $\mu$ a strictly meromorphic partition of $2g-2$ of length $n$ such that $\mathcal{H}_g(\mu)$ is non-empty. Then for every subset $A \subseteq \left\{1,2,\ldots,n \right\}$ of cardinality $|A| \leq g$ there exists an $n$-pointed curve $[C,p_1,\ldots,p_n] \in \mathcal{H}_g(\mu)$ such that 
	\[h^0(C,\sum_{i\in A} p_i)= 1\]    

\end{prop}
\begin{proof}
	We prove this result by induction on the genus. The proposition is clear for genus $g=1$ and the case $g=2$ is treated in Proposition \ref{k=1g=2}, hence we assume $g\geq 3$.
	Using Proposition 4.20 from \cite{Barthesis} and Corollary \ref{cor} we observe that the proposition is true when $n\geq g$. In particular, we can assume without loss of generality that $n\leq g-1$ and $A = \left\{1,2,\ldots,n\right\}$. As  previously explained, what we need to prove is equivalent to the surjectivity of the map 
	 \[H^0(C,\omega_C) \rightarrow H^0(C,\omega_C/ \omega_C(-\sum\limits_{i\in A} p_i))\] \par 
	 By Lemma \ref{lmdeg}, it is enough to find a pointed curve $[C = C_1\cup C_2, p_1,\ldots,p_n] \in \overline{\mathcal{H}}_g(\mu)$ having an unique node and satisfying the conditions: 
\[h^0(C_1,\mathcal{O}_{C_1}(\sum_{i\in A_1} p_i)) = 1 \text{ \ and \ } h^0(C_2,\mathcal{O}_{C_2}(\sum_{i\in A_2} p_i)) = 1 \]
where $A_i = \left\{ j\in A \mid p_j \in C_i \right\}$ \par 
In order to see that such a curve exists in the compactification, observe that there exists an $m_i \geq 3 $ and use the clutching:  
\[ \mathcal{H}_1(m_i, -m_i) \times \mathcal{H}_{g-1}(m_1,\ldots,\hat{m}_i,\ldots,m_n,m_i-2) \rightarrow \overline{\mathcal{M}}_{g,n} \] \par
Here we glue along the last markings of the two strata. By the description of the compactification in \cite{DaweiAbComp} we conclude that the image of the clutching is contained in the closure of $\mathcal{H}_g(\mu)$. The conclusion follows by the induction hypothesis. 

\end{proof}
We say that a partition $\mu$ of $2g-2$ is of even type if all its entries are even or if it is of the form $(2m_1,\ldots,2m_n,-1,-1)$ with $m_i$ positive for all $i$.
We study the case when $\mu$ is of even type and deduce a similar statement as the one before. This will require a little more care, as we would need to use a degeneration argument that keeps track of the parity of the spin structure. 
 \begin{prop} \label{odev} 
 	Let $\mu$ a partition of $2g-2$ of length $n$ and even type. Take $Z$ a connected component of $\mathcal{H}_g(\mu)$ and $A \subseteq \left\{1,2,\ldots,n \right\}$ a subset of cardinality $|A| \leq g $. Assume we are not in one of the following exceptional cases:  \\ 
 	i)  $Z$ is a hyperelliptic component \\ 
 	ii) $Z=\mathcal{H}^{\mathrm{even}}_g(2,2,\ldots,2,-1,-1)$ and $A = \left\{1,2,\ldots,g \right\}$ \\  
    iii) $Z=\mathcal{H}^{\mathrm{odd}}_g(2,2,\ldots,2,-2)$ and $A = \left\{1,2,\ldots,g \right\}$ \par
    Then there exists a pointed curve $[C,p_1,\ldots,p_n] \in Z$ such that: 
    \[h^0(C,\mathcal{O}_C(\sum_{i\in A} p_i)) = 1\] 
 \end{prop}
\begin{proof}
	We proceed as in the previous proof. We first observe that the statement is obviously true for the case $g=1$ and assume it is true by induction for the case of genus up to $g-1$. Without any loss of generality we can assume $|A| = \min\left\{g,n\right\}$. \par 
	In the cases ii) and iii), the proposition is true for the other connected component by Proposition \ref{gen}, so we can assume we are not in this setting for $\mu$ and $A$. \par
	
	As we excluded cases i-iii) for $Z$, we are in one of the following three cases: \\ 
	a) There exists $i \in A$ such that $m_i \geq 4$ \\
	b) $g=2$ with $\mu = (4,-1,-1)$ and $A =\left\{2,3\right\}$ or $\mu =(2,2,-1,-1)$ and $A = \left\{ 3,4\right\}$ \\
	c) There exist $i \in A$ and $B \subseteq \left\{1,2,\ldots,n\right\}$ in the complement of $A$ such that $m_i$ is even, $m_i+\sum_{j \in B} m_j \geq 4$ and $m_j \geq 2$ for all $j\in B$. \\
	
	Notice first that case b) is covered by Proposition \ref{k=1g=2}, as the strata have two connected components and the statement is false for the hyperelliptic one. 
	 For the other cases we take the following clutchings where the glueings are along the last markings. They correspond to the cases a) and c) respectively: 
	\[ \mathcal{H}_1(m_i,-m_i) \times \mathcal{H}_{g-1}(m_1,\ldots,\hat{m}_i,\ldots,m_n,m_i-2) \rightarrow \overline{\mathcal{H}}_g(\mu) \subseteq \overline{\mathcal{M}}_{g,n}\]
	\[\mathcal{H}_1(m_i, \left\{m_j\right\}_{j \in B}, -m_i-\sum_{j\in B}m_j) \times \mathcal{H}_{g-1}(\left\{m_j\right\}_{j \notin \left\{ i \right\} \cup B}, m_i+\sum_{j\in B}m_j - 2) \rightarrow \overline{\mathcal{H}}_g(\mu)\subseteq \overline{\mathcal{M}}_{g,n} \] \par
	The strata in genus 1 on the left of the clutchings have both even and odd components and the strata in genus $g-1$ on the right satisfy the induction hypothesis for at least one component by Proposition \ref{gen}. The conclusion follows as long as we can obtain singular curves as in Lemma \ref{lmdeg} that lay in the even and odd components respectively. Using the parity description for spin structures on curves with a unique node in \cite{corn}, this is clearly the case when $\mathcal{H}_g(\mu)$ has exactly two connected components that are distinguished by parity. This is the case when $\mu$ is not of hyperelliptic type or when $g=2$ and $\mu$ is hyperelliptic. \\
	
	As the genus 2 case is completely covered, we assume $g\geq 3$. The proposition is also true in the cases $n=3,4$ with $\mu$ of hyperelliptic type, as otherwise the hyperelliptic component would have expected image dimension. 
	We are left with the cases $\mu$ of hyperelliptic type and $n=1$ or $2$. The case of $n=1$ is obvious and we now treat the second case: 
	Assume that a component $Z$ of $\mathcal{H}_g(2n,-2m)$ that is not hyperelliptic satisfy $h^0(C,p_1+p_2)=2$ for all points $[C,p_1,p_2] \in Z$. It follows that $C$ is hyperelliptic and $p_1+p_2$ is a $g^1_2$. By hypothesis we have \[2np_1-2mp_2 = (g-1)g^1_2\]
	Adding $2mp_1+2mp_2 = 2mg^1_2$ to both sides we get
	\[2(n+m)p_1=(g-1+2m)g^1_2 \] \par
	Hence we have a finite number of choices for $p_1$ and $p_2$. From this it follows that the map 
	$\pi_Z$ is finite, contradicting the assumption because of Corollary \ref{cor}

\end{proof}  \ \par
It is useful to observe that the proposition is true for both components when $\mu = (2,2,\ldots,2,-1,-1)$ or $(2,2,\ldots,2,-2)$ and $A= \left\{1,2,\ldots,g-1,g+1\right\}$. \par 
We now explain how Proposition \ref{gen} and Proposition \ref{odev} are sufficient to deduce Theorem \ref{T1}. This will follow as a trivial application of the Riemann-Roch Theorem. \\

\textbf{Proof of Theorem \ref{T1}:} Take $Z$ a nonhyperelliptic component of $\mathcal{H}_g(\mu)$. Using Propositions \ref{gen} and \ref{odev} it follows that we can find $A \subseteq \left\{1,2,\ldots,n\right\}$ of cardinality $\min\left\{g,n\right\}$ and a smooth pointed curve $[C,p_1,\ldots,p_n]$ in $Z$ 
such that 
\[h^0(C, \sum_{i\in A} p_i) = 1\]
It then follows that:
\[h^0(C,\sum_{i=1}^n p_i) = \max\left\{1,n-g+1\right\}\] \par
Theorem \ref{T1} follows from Corollary \ref{cor}. \hfill $\qed$
\section{Quadratic strata } \ \par 
There is nothing fundamentally different in this case from the meromorphic one. However, this case has its particularities as there is no parity to be taken into account, there are some sporadic strata that have to be considered separately and there are some empty strata in low genus that need to be accounted when applying the induction. All this information can be found in \cite{ChenM} and \cite{Lanneau}. We proceed to prove the analogue in the quadratic case of Proposition \ref{gen}, starting with the base case $g=2$.    
\begin{prop} \label{k=2g=2}
	Let $\mu = (m_1,\ldots,m_n)$ be a partition of $4$ with all entries either -1 or positive, such that $\mathcal{Q}_2(\mu)$ is non-empty. Then for every subset $A \subseteq \left\{1,2,\ldots,n\right\}$ of cardinality $|A| \leq 2$ there exists an element $[C,p_1,\ldots,p_n] \in \mathcal{Q}_2(\mu)$ such that: 
	\[h^0(C, \sum_{i\in A}p_i) = 1\]
	with the exception of the case $\mu = (2,1,1)$ and $A = \left\{2,3 \right\}$ when the only component is the hyperelliptic one.
\end{prop}
\begin{proof}
	The case $n=1$ is obvious and we start by considering the case $n=2$. Assume that a generic element $[C,p,q]$ in a stratum $\mathcal{Q}_2(m,4-m)$ satisfies $p+q = g^1_2$. It then follows from $mp + (4-m)q = 2g^1_2$ that $(2m-4)p =(m-2)g^1_2$. As a consequence over a curve $C$ there are finitely many choices for $p$ and $q$. Corollary \ref{cor} contradicts our assumption. 
	
	We can thus assume that
	$n \geq 3$. Take $\mu = (m_1,\ldots,m_n)$ a partition of 4 and assume that for all points $[C,p_1,p_2,\ldots,p_n] \in \mathcal{Q}_2(\mu)$ we have $p_1+p_2 = g^1_2$. \par	If one of the entries $m_1,m_2$ is at least 3, say $m_1$, then we can use the clutching along the last markings: 
	\[ \mathcal{Q}_1(m_1,-m_1) \times \mathcal{Q}_1(m_2,\ldots,m_n,m_1-4) \rightarrow \overline{\mathcal{M}}_{2,n} \] 
	The image of this map is in $\overline{\mathcal{Q}}_2(\mu)$ by \cite{Daweik-diffcomp} and hence by Lemma \ref{lmdeg} we derive a contradiction by induction as long as the stratum on the right is nonempty. As we assumed that $n \geq 3$, the strata is empty only for the cases $\mu = (4,1,-1)$ and $\mu = (4,-1,1)$. These cases can be treated by hand. \par
	If $\mu = (4,1,-1)$, by our assumption it follows that $3p_1-p_3 = g^1_2$ and hence $p_1$ is a Weierstrass point and $p_2 = p_1$, contradicting the fact that the markings are distinct. \par 
	If $\mu = (4,-1,1)$ it follows that $5p_1+p_3 = 3g^1_2$. Hence the image of the map forgetting the marking $p_2$ is contained in $\mathcal{H}_2^3(5,1)$. It follows that we have a finite map 
	\[ \mathcal{Q}_2(4,-1,1)\rightarrow \mathcal{H}_2^3(5,1)\] \par
	This contradicts the fact that the dimension of the space on the left is 4 while the space on the right has dimension 3. \\
	
	We are thus left with the case when both entries $m_1$ and $m_2$ are strictly less than 3. \par 
	Take the case $m_1= m_2 = m$. Then by forgetting the first two markings, we get a map 
	\[ \mathcal{Q}_2(\mu) \rightarrow \mathcal{H}_2^{2-m}(m_3,\ldots,m_n)  \] 
	with generically one dimensional fibers. By dimension reasons, the only case when this is possible is when the image corresponds to a component of holomorphic abelian differentials. This never happens when $m=2$ and can happen only in the cases $\mu = (1,1,2)$ and $(1,1,1,1)$ when $m=1$ and in the cases $\mu = (-1,-1,3,3)$ or $(-1,-1,6)$ when $m=-1$. \\
	
	When $m=1$, the first case is the exception already mentioned in the proposition and for the second case,  we take a curve $[C,p,q,p',q'] $ in $\mathcal{Q}_2(1,1,1,1)$ where $p,q$ are conjugate with $p'$ and $q'$ respectively. Hence our assumption was wrong and there exist a curve with $p_1+p_2 \neq g^1_2$, except for the known exception $\mu = (1,1,2)$. \\ 
	
	In the two cases when $m=-1$, the contradiction follows as we know from \cite{Lanneau} that these strata also have a nonhyperelliptic component, that does not map to the strata $\mathcal{H}_2(1,1)$ and $\mathcal{H}_2(2)$ respectively. \\ 
	
	The only case left is when $m_1$ and $m_2$ are distinct and less than $3$. Assume $m_1 > m_2$. As they cannot be both equal to 2, it follows that there must be another positive entry, say $m_3$. By our assumption $p_1+p_2 = g^1_2$ for all points $[C,p_1,\ldots,p_n]$ in the stratum. We see that by forgetting the point $p_1$ we get a finite map: 
	\[\mathcal{Q}_2(\mu) \rightarrow \mathcal{H}_2^{2-m_1}(m_2-m_1,m_3,\ldots,m_n)\] \par 
	There is no component corresponding to holomorphic abelian differentials as there are both positive and negative entries. Hence we obtain a contradiction by dimension reasons. \par
	Our proof by contradiction is now complete. 
	
\end{proof}

We can extend this result for all genera and we have the following statement.
\begin{prop} \label{quad}
	Let $\mu = (m_1,\ldots,m_n)$ be a partition of $4g-4$ with all entries either -1 or positive, such that $\mathcal{Q}_g(\mu)$ is non-empty. Then for every subset $A \subseteq \left\{1,2,\ldots,n\right\}$ of cardinality $|A| \leq g$ there exists an element $[C,p_1,\ldots,p_n] \in \mathcal{Q}_g(\mu)$ such that: 
	\[h^0(C, \sum_{i\in A}p_i) = 1\]
	with the exception of the case $\mu = (2,1,1)$ and $A = \left\{2,3 \right\}$ when the only component is the hyperelliptic one.
\end{prop}
\begin{proof}
	We first observe that for any partition $\mu$ of 0 different from $(0)$ and $(1,-1)$, the stratum $\mathcal{Q}_1(\mu)$ is non-empty. Again, we consider this case to be the initial step and proceed by induction on the genus $g$. The genus 2 case is Proposition \ref{k=2g=2}. \\
	
	Suppose that the proposition is true for genus up to $g-1$ and let us prove it for $g \geq 3$. Again, by proposition 4.20 in \cite{Barthesis} we see that the proposition is true when $n \geq g$ and hence it is enough to treat the case $n\leq g-1$ and $A = \left\{1,2,\ldots,n\right\} $. \par 
	In this case it follows that there exists $m_i \geq 4$ and without loss of generality we assume that $i=1$. Then we take the clutching along the last markings
	\[ \mathcal{Q}_1(m_1,-m_1)\times \mathcal{Q}_{g-1}(m_2,\ldots,m_n,m_1-4) \rightarrow \overline{\mathcal{M}}_{g,n} \] 
	
	By \cite{Daweik-diffcomp}, the image is in the closure of $\mathcal{Q}_g(\mu)$. This map is well-defined as long as both the strata in the clutching are non-empty. 
	If the stratum in genus $g-1$ on the right is non-empty and not $\mathcal{Q}_2(1,1,2)$, we can apply the induction hypothesis as in Proposition \ref{gen} and the conclusion follows. \\
	
	 The cases when we cannot apply induction are the cases when the genus is 3 and the partition $\mu$ is $(8),(4,4), (5,3),(7,1),(4,3,1)$ or $(6,1,1)$. The cases $(8)$ and $(4,4)$ are obvious and the cases $(5,3)$ and $(4,3,1)$ follow by using another clutching, hence the only cases left to treat are $(7,1)$ and $(6,1,1)$. When $\mu =(7,1)$ the method in the proof of the case $n=2$ in Proposition \ref{k=2g=2} implies the conclusion.  

In the case $(6,1,1)$, assume that the forgetful map to $\mathcal{M}_3$ has generically 1-dimensional fibers. Take a point $[C,p,q,r] \in \mathcal{Q}_3^{\mathrm{nonhyp}}(6,1,1)$. As we have $h^0(C,2K_C-6p) =h^0(C,q+r) \geq 1$, it follows that $p$ is a 2-fold Weierstrass point. Hence the image of the map 
\[\mathcal{Q}_3^{\mathrm{nonhyp}}(6,1,1)\rightarrow \mathcal{M}_{3,1}\]
forgetting $q$ and $r$ is generically finite over the image in $\mathcal{M}_3$. By our assumption, the map forgetting the sections $q$ and $r$ has generically 1-dimensional fibers and hence by Corollary \ref{cor} we have $h^0(C,q+r)=2$. \par 
In this case $C$ is hyperelliptic and $6p = 3g^1_2$. 
Denote by $p'$ the conjugate of $p$. Then we have $3p=3p'$ and hence $h^0(3p) = 2$, which is possible if and only if $p$ is a Weierstrass point. Then $[C,p,q,r] \in \mathcal{Q}_3^{\mathrm{hyp}}(6,1,1)$, contradicting the fact that the components are disjoint. 

Consequently, the proposition is true for all partitions.
\end{proof} 

We proceed to prove Theorem \ref{T2}.

\textbf{Proof of Theorem \ref{T2}:} Analogously to the proof of Theorem \ref{T1}, it follows from Proposition \ref{quad} that every stratum has a connected component of the expected image dimension. In particular, the only cases left to consider are those when $\mathcal{H}_g(\mu)$ has more than one connected component. \par
 The cases when $n\geq 3$ and $\mu$ is hyperelliptic also follow as the hyperelliptic component is not of the expected image dimension. \\
 
 The only cases that are left to study are those when $\mu$ is of hyperelliptic type with $n=2$ or one of the sporadic cases $(9,-1), (6,3,-1),(3,3,3,-1),(12),(9,3),(6,6),(6,3,3)$ and $(3,3,3,3)$ when the stratum has two connected components. When $n=2$ it follows that both components are of the expected image dimension by applying the same method as for the case $n=2$ in Proposition \ref{k=2g=2}. The case with $n=1$ is also obvious. In particular, we need to prove Theorem \ref{T2} only in the cases $(6,3,-1), (3,3,3,-1),(6,3,3)$ and $(3,3,3,3)$. We will use the description of the strata appearing in \cite{ChenM} to draw the conclusion. \\
 
 We claim that Theorem \ref{T2} holds for these sporadic strata and hence the conclusion follows. \hfill $\qed$

\subsection{Sporadic quadratic strata} \ \par
It is fruitful to consider for these cases a proof by contradiction. This is due to Corollary \ref{cor}, as an assumption on the dimension of the fibers provides an extra divisorial information for the marked points. \\

\textbf{Case 1: $\mu = (6,3,-1)$}. Take $Z$ one of the two connected components of $\mathcal{Q}_3(6,3,-1)$ and assume that the map 
\[Z\rightarrow \mathcal{M}_3\]
is not generically finite. Using Corollary \ref{cor} we have $h^0(C,p+q+r) = 2$  for every point $[C,p,q,r] \in Z$. \par
Using the Riemann-Roch Theorem we deduce that $h^0(C,K_C-p-q-r) = 1$. In particular, there exists a point $s \in C$ such that: 
\[p+q+r+s = K_C\] \par 
Using this relation, together with $6p+3q-r = 2K_C$ we can deduce: 
\[4r+3s-3p = K_C\] \par 
There are three possible cases: \\ 
i) Generically $s=p$, case in which we have $4r=K_C$ and hence we have finite number of choices for $r$ and by our assumption, the composition map 
\[\mathcal{Q}_3(6,3,-1) \rightarrow\mathcal{H}_3(4) \rightarrow \mathcal{M}_3\] 
has one-dimensional fibers, and hence by Corollary \ref{cor} we have $p+q = g^1_2$. \par 
It follows that $3p-r = g^1_2$ and hence $p$ is a Weierstrass point. Then $p=q$, contradicting the fact that they are distinct. \\ 
ii) Generically $s=r$ and then the map from $\mathcal{Q}_3(6,3,-1) \rightarrow \mathcal{M}_{3,2}$ forgetting the marking $q$ has image $\mathcal{H}_3(7,-3)$ which is of the expected image dimension by Theorem \ref{T1}.  \\
iii) The image of the map $\mathcal{Q}_3(6,3,-1) \rightarrow \mathcal{M}_{3,2}$ forgetting the marking $q$ is the image of the forgetful map $\mathcal{H}_3(4,3,-3) \rightarrow \mathcal{M}_{3,2}$ forgetting the second marking. As $\mathcal{H}_3(4,3,-3)$ is generically finite over $\mathcal{M}_3$ from Theorem \ref{T1}, it follows that $Z$ is generically finite over $\mathcal{M}_3$, contradicting our assumption. \\

Since all possible cases yield the same result, we conclude that the connected components of $\mathcal{Q}_3(6,3,-1)$ have image of dimension $2g-3+n = 6$ in $\mathcal{M}_3$. \\

\textbf{Case 2:} $\mu = (3,3,3,-1)$. We take $Z$ a connected component of $\mathcal{Q}_3(3,3,3,-1)$ and assume that the map 
\[\pi_Z\colon Z \rightarrow \mathcal{M}_3 \]
does not have one dimensional fibers as expected. It follows that it must have two dimensional fibers and then, by Corollary \ref{cor} we have that for all points $[C,p,q,r,s] \in Z$ we have: 
\[h^0(C,p+q+r+s) = 3\] \par
By Clifford's Theorem, it follows that $C$ is hyperelliptic and $p+q+r+s = 2g^1_2$. Without loss of generality assume that $p,q$ and $r,s$ are pairs of conjugate points for the hyperelliptic involution. It then follows that $3r-s = g^1_2$ and hence $r$ is a Weierstrass point, hence $r=s$, contradicting the fact that they are distinct.  \\ 

\textbf{Case 3:} $\mu = (3,3,3,3)$. In this case, according to \cite{ChenM} the two non-hyperelliptic connected components are differentiated by whether the value of $h^0(C,p_1+p_2+p_3+p_4)$ is 1 or 2. This value is exactly what we need to deduce the dimension of the generic fiber. Hence one of the two components is dominant over $\mathcal{M}_4$ while the other, denoted $\mathcal{Q}_4^{\mathrm{Irr} }(3,3,3,3)$ has generically one dimensional fibers. \par 
\begin{rmk} \label{nicediv} \normalfont
	The image of the component $\mathcal{Q}_4^{\mathrm{Irr}}(3,3,3,3)$ provides an interesting divisor on the moduli space $\mathcal{M}_4$, namely the locus of curves $[C]$ that have a $g^1_4$ denoted by $A$ with the property that $3A=2K_C$. By observing that $h^0(K_C-A) = 1$ we deduce that the locus is also the image of the map $\mathcal{H}_4^{\mathrm{nonhyp}}(3,3) \rightarrow \mathcal{M}_4$. The class in $\mathrm{Pic}(\overline{\mathcal{M}}_4)\otimes \mathbb{Q}$ of its closure was computed in \cite{Mullane}.
\end{rmk}  \par
\textbf{Case 4:} $\mu = (6,3,3)$. Again, the two non-hyperelliptic connected components are differentiated by whether the value of $h^0(C,2p_1+p_2+p_3)$ is 1 or 2. \par 
For the component with $h^0(C,2p_1+p_2+p_3) = 1$ it follows that $h^0(C,p_1+p_2+p_3) = 1$ and hence the projection to $\mathcal{M}_4$ is generically finite. \par 
In the case when $h^0(C,2p_1+p_2+p_3) = 2$ we observe that for every curve in the divisor in Remark \ref{nicediv} there exist $p_1,p_2,p_3 \in C$ such that $2p_1+p_2+p_3 = A$. As both the component and the divisor are irreducible of the same dimension and the divisor is in the image of the projection to $\mathcal{M}_4$ it follows that it is the image of the forgetting map and hence the map is generically finite. \par

\bibliographystyle{alpha}

\bibliography{main}

HUMBOLDT UNIVERSIT\"AT ZU BERLIN, INSTITUT F\"UR MATHEMATIK, RUDOWER CHAUSEE 25, 12489 BERLIN, GERMANY \par 
E-mail address: andreibud95@protonmail.com

\end{document}